\documentclass[12pt,final]{article}
\usepackage[margin=2.5cm,top=1.5cm]{geometry}
\usepackage[all,arc]{xy}
\usepackage{mathrsfs}
\usepackage{amsmath,amsthm,amssymb,enumerate}
\usepackage{color}

\newtheorem{definition}{Definition}[section]
\newtheorem{lemma}[definition]{Lemma}
\newtheorem{theorem}[definition]{Theorem}
\newtheorem{corollary}[definition]{Corollary}

\newtheorem{proposition}[definition]{Proposition}

\newtheorem{remark}[definition]{Remark}

\newcommand{\Romannum}[1]{\uppercase\expandafter{\romannumeral #1}}

\numberwithin{equation}{section}
\newcommand\keywordsname{Key words}
\newcommand\AMSname{AMS subject classifications}

\newenvironment{@abssec}[1]{%
     \if@twocolumn
       \section*{#1}%
     \else
       \vspace{.05in}\footnotesize
       \parindent .2in
         {\upshape\bfseries #1. }\ignorespaces
     \fi}
     {\if@twocolumn\else\par\vspace{.1in}\fi}

\begin{document}

\title{On the similarity of Tensors\footnote{P. Yuan's research is supported by the NSF of China (Grant No. 11271142) and
the Guangdong Provincial Natural Science Foundation(Grant No. S2012010009942),
L. You's research is supported by the Zhujiang Technology New Star Foundation of
Guangzhou (Grant No. 2011J2200090) and Program on International Cooperation and Innovation, Department of Education,
Guangdong Province (Grant No.2012gjhz0007).}}
\author{Pingzhi Yuan\footnote{{\it{Email address:\;}}yuanpz@scnu.edu.cn.}, Lihua You\footnote{{\it{Corresponding author:\;}}ylhua@scnu.edu.cn.}}
\vskip.2cm
\date{{\small
School of Mathematical Sciences, South China Normal University,\\
Guangzhou, 510631, P.R. China\\
}}
\maketitle
%\vskip20pt

\begin{abstract}
In this paper, we characterize all similarity relations when $m\ge3$,
 obtain some interesting properties which are different from the case $m=2$,
and show that  the results of matrices about the Jordan canonical form  cannot be extended to tensors.
%\hskip.6cm Let $\mathbb{P}_n$ be the set of all matrices which have the same zero patterns with some permutation matrix of order $n$.
% In this paper, we prove the following result:
%Let $\mathbb{I}$ be the unit tensor of order $m\ge3$ and dimension $n\ge2$. Suppose that $P$ and $Q$ are two matrices with $P\mathbb{I}Q=\mathbb{I}$,
%then $P,Q\in \mathbb{P}_n$. This gives a characterization for the similarities of tensors with order $m\ge3$.
\end{abstract}

\section{Introduction}
\hskip.6cm Since the work of Qi \cite{Qi05} and Lim \cite{Li05}, the study  of tensors and the spectra of tensors (and hypergraphs) and their various applications have attracted much attention and interest.

An order $m$ dimension $n$ tensor $\mathbb{A}= (a_{i_1i_2\ldots i_m})_{1\le i_j\le n \hskip.2cm (j=1, \ldots, m)}$ over the complex field $\mathbb{C}$ is a multidimensional array with all entries $a_{i_1i_2\ldots i_m}\in\mathbb{C}\, ( i_1, \ldots, i_m\in [n]=\{1, \ldots, n\})$. The majorization matrix  $M(\mathbb{A})$ of the tensor $\mathbb{A}$ is defined as  $(M(\mathbb{A}))_{ij}=
a_{ij\ldots j}, (i, j\in[1, n])$ by Pearson \cite{Pe10}.  The unit tensor of order $m$ and dimension $n$ is the tensor $\mathbb{I}=(\delta_{i_1, i_2, \ldots, i_m}) $ with
 $$\delta_{i_1, i_2, \ldots, i_m}=\left\{
   \begin{array}{cc}
   1, & { \rm if \hskip.2cm }  i_1=i_2=\ldots=i_m; \\
    0,  &{\rm otherwise}.
   \end{array} \right.$$

 Let $\mathbb{A}$ (and $\mathbb{B}$) be an order $m\ge2$ (and $k\ge 1$), dimension $n$ tensor, respectively. Recently, Shao \cite{Sh12} defined a general product  $\mathbb{A}\mathbb{B}$ to be the following tensor $\mathbb{D}$ of order $(m-1)(k-1)+1$ and dimension $n$:
$$ d_{i\alpha_1\ldots\alpha_{m-1}}=\sum\limits_{i_2, \ldots, i_m=1}^na_{ii_2\ldots i_m}b_{i_2\alpha_1}\ldots b_{i_m\alpha_{m-1}} \quad (i\in[n], \, \alpha_1, \ldots, \alpha_{m-1}\in[n]^{k-1}).$$

The tensor product  possesses a very useful property: the associative law (\cite{Sh12}, Theorem 1.1). With the general product,  the following definition of the similarity relation of two tensors was proposed by Shao \cite{Sh12}.

\begin{definition}{\rm (\cite{Sh12}, Definition 2.3)} \label{defn11} Let $\mathbb{A}$ and $\mathbb{B}$ be two order $m$ dimension $n$ tensors. Suppose that there exist two matrices $P$ and $Q$ of order $n$ with $P\mathbb{I}Q=\mathbb{I}$ such that $\mathbb{B}=P\mathbb{A}Q$, then we say that the two tensors are similar.\end{definition}

It is easy to see that the similarity relation is an equivalent relation,  and similar tensors have the same characteristic polynomials, and thus they have the same spectrum (as a multiset). For example, the permutation similarity and the diagonal similarity (see also \cite{Sh12, YY10, YY11}) are
 two special kinds of the similarity of tensors.

\begin{definition} {\rm (\cite{Sh12})} \label{defn12}  Let $\mathbb{A}$ and $\mathbb{B}$ be two order $m$ dimension $n$ tensors.
We say that $\mathbb{A}$ and $\mathbb{B}$  are permutational similar,
if there exists some permutation matrix $P$ of order $n$ such that  $\mathbb{B}=P\mathbb{A}P^T$,
where $\sigma\in S_n$ is a permutation on the set $[n]$ and $P=P_{\sigma}=(p_{ij})$ is the corresponding permutation matrix of $\sigma$ with $p_{ij}=1\Leftrightarrow j=\sigma(i)$.\end{definition}

\begin{definition} {\rm (\cite{Sh12}, Definition 2.4)} \label{defn13} Let $\mathbb{A}$ and $\mathbb{B}$ be two order $m$ dimension $n$ tensors.
We say that $\mathbb{A}$ and $\mathbb{B}$  are diagonal similar,
if there exists some invertible diagonal matrix $D$ of order $n$ such that  $\mathbb{B}=D^{-(m-1)}\mathbb{A}Q$.\end{definition}

 About the matrices $P$ and $Q$ in Definition \ref{defn11}, \cite{Sh12} showed the following proposition.

\begin{proposition}\label{prop1}{\rm (\cite{Sh12}, Remark 2.1)}
If $P$ and $Q$ are two matrices of order $n$ with $P\mathbb{I}Q=\mathbb{I}$, where $\mathbb{I}$ is the order $m$ dimension $n$
unit tensor, then both $P$ and $Q$ are invertible matrices.
\end{proposition}

In this paper, we characterize all similarity relations when $m\ge3$ in Section 2,
and in Section 3, we obtain some interesting properties which are different from the case $m=2$,
and show that  the results of matrices about the Jordan canonical form  cannot be extended to tensors.

\section{Main results}

\hskip.6cm In this section, we will prove Theorem \ref{thm21}. We first prove the following two lemmas.

\begin{lemma}\label{lem1}
Let  $P=(p_{ij})$ and $Q$ be two matrices of order $n$ with $P\mathbb{I}Q=\mathbb{I}$,
$\mathbb{I}$ the order $m$ dimension $n$ unit tensor.
Take $\mathbb{I}Q=\mathbb{A}= (a_{i_1i_2\ldots i_m})$, then we have

{\rm (i) } $a_{i_1i_2\ldots i_m}=0$ when $(i_2, i_3, \ldots, i_m)\ne (i_2, i_2, \ldots, i_2)$;

{\rm (ii) } $PM(\mathbb{A})=I$, where $ M(\mathbb{A})$ is the majorization matrix of $\mathbb{A}$ and $I$ is the unit matrix of order $n$.
\end{lemma}

\begin{proof} For the proof of (i), we let $\alpha=i_2i_3\ldots i_m\ne i_2i_2\ldots i_2$. By $P\mathbb{A}=\mathbb{I}$, we have
$$\delta_{i\alpha}=\sum_{j=1}^n p_{ij}a_{j\alpha}, \, i=1, 2, \ldots, n.$$
Since $\alpha\ne i_2i_2\ldots i_2$, then $\delta_{i\alpha}=0$, it follows that
$$P(a_{1\alpha}, a_{2\alpha}, \ldots, a_{n\alpha})^T=(0, 0, \ldots, 0)^T.$$
Observe that $P$ is an invertible matrix by Proposition \ref{prop1}, thus $a_{j\alpha}=0, j=1, 2, \ldots, n$, and the conclusion of (i) follows.

For the proof of (ii), by $P\mathbb{A}=\mathbb{I}$, we have
$$\delta_{ij\ldots j}=\sum_{u=1}^n p_{iu}a_{uj\ldots j}=\sum_{u=1}^n p_{iu}(M(\mathbb{A}))_{uj}=(PM(\mathbb{A}))_{ij}.$$
Hence $PM(\mathbb{A})=I$, this completes the proof of (ii).\end{proof}

%To complete the proof of Theorem 1.2, we need to prove the following lemma.

\begin{lemma}\label{lem2}
Let  $\mathbb{I}$ be the  order $m\ge3$ dimension $n$  unit tensor. Suppose that  $Q=(q_{ij})$ is a matrix of order $n$ such that
$$(\mathbb{I}Q)_{i\alpha}=0$$ for all $i\in[n]$ and all $\alpha\ne j\ldots j, j\in [n]$, then there is at most one nonzero element in every row of $Q$.\end{lemma}
\begin{proof} Since
$$(\mathbb{I}Q)_{i_1i_2\ldots i_m}=\sum_{j_2, \ldots, j_m=1}^n\delta_{i_1j_2\ldots j_m}q_{j_2i_2}\ldots q_{j_mi_m}=q_{i_1i_2}q_{i_1i_3}\ldots q_{i_1i_m},$$
 by the assumption, for every $\alpha=i_2i_3\ldots i_m\ne i_2i_2\ldots i_2$, we get
\begin{equation}\label{eq1}
q_{i_1i_2}q_{i_1i_3}\ldots q_{i_1i_m}=0.
\end{equation}

 If $q_{i_1t}\ne0$ for some $t\in[n]\setminus\{i_2\}$, then we take $i_3=\ldots =i_m=t$. By (\ref{eq1}) we have $q_{i_1i_2}q_{i_1t}^{m-2}=0$, hence $q_{i_1i_2}=0$. Since the choice of $i_2$ is arbitrary, then we have proved that for any $i\in[n]$, there is at most one nonzero element in $\{q_{i1}, q_{i2}, \ldots, q_{in}\}$. This completes the proof of  the lemma.\end{proof}

\begin{theorem}\label{thm21} Let $\mathbb{I}$ be the unit tensor of order $m\ge3$ and dimension $n\ge2$.
Suppose that $P$ and $Q$ are two matrices of order $n$ with $P\mathbb{I}Q=\mathbb{I}$,
then there exist a permutation matrix $R_{\sigma}$ and an invertible diagonal matrix $D$,
such that $Q=DR_{\sigma}$ and $P=R_{\sigma}^TD^{1-m}$,
where $\sigma\in S_n$ is a permutation on the set $[n]$. \end{theorem}

\begin{proof} Since $\mathbb{I}=P\mathbb{I}Q=P(\mathbb{I}Q)$, combining Lemmas \ref{lem1} and  \ref{lem2}, we obtain that there is at most one nonzero element in every row of $Q$. Note that $Q$ is an invertible matrix by Proposition \ref{prop1}, hence  there is precisely one nonzero element, say, $q_{i\sigma(i)}\ne0, 1\le i\le n$ in every row of $Q$.
Obviously, $\sigma(1), \sigma(2), \ldots, \sigma(n)$ is a permutation of $\{1, 2, \ldots, n\}$, hence
$$Q=DR_\sigma$$ is the product of   an invertible diagonal matrix $D=diag(d_{11}, \ldots, d_{nn})$, where $d_{ii}=q_{i\sigma(i)}$ and  a permutation matrix $R_\sigma=(r_{ij})$, where $r_{ij}=1$  if $j=\sigma(i)$ and $r_{ij}=0$ otherwise. %Thus $Q\in \mathbb{P}_n$.

Now by the general tensor product, we get
$$(M(\mathbb{I}Q))_{ij}=(\mathbb{I}Q)_{ij\ldots j}=q_{ij}^{m-1}=\left\{
   \begin{array}{cc}
   q_{i\sigma(i)}^{m-1} & { \rm if \hskip.2cm }  j=\sigma(i); \\
    0  &{\rm otherwise}.
   \end{array} \right.$$
Hence $M(\mathbb{I}Q)$ is the product of the invertible diagonal matrix $D^{m-1}=diag(d_{11}^{m-1}, \ldots, d_{nn}^{m-1})$ and the permutation matrix $R_\sigma=(r_{ij})$. By Lemma \ref{lem1} (ii), we have $PM(\mathbb{I}Q)=I$, it follows that $P=R_\sigma^TD^{1-m}$, is the product of the permutation matrix $R_\sigma^T$ and the invertible diagonal matrix $D^{1-m}=diag(d_{11}^{1-m}, \ldots, d_{nn}^{1-m})$. %Thus $P\in \mathbb{P}_n$.
\end{proof}

\begin{remark}
By the proof of Theorem \ref{thm21}, we see that $Q$ and $P$ are closely related with $QP=D^{2-m}$ and $PQ=R_{\sigma}^TD^{2-m}R_{\sigma}$ are invertible diagonal matrices.  Note that when $m=2$, $P, Q$ are invertible matrices and $PQ=QP=I$. It implies that the case of $m\geq 3$ is completely different from the case $m=2$.
\end{remark}

Let $\mathbf{P}_n$ be the set of all permutation matrices of order $n$,
$\mathbb{P}_n$ be the set of all matrices which have the same zero patterns with some permutation matrix of order $n$.
Clearly, $\mathbf{P}_n\subseteq\mathbb{P}_n$. For example,
$ S=\left(
  \begin{array}{cc}
    0 & 1 \\
    1 & 0\\
  \end{array}
\right)\in \mathbf{P}_2, T=\left(
                                                  \begin{array}{cc}
                                                    0 & 2 \\
                                                   -3 & 0 \\
                                                  \end{array}
                                                \right)\in \mathbb{P}_2$,
                                                where $S$ and $T$ have the same zero pattern.

\begin{remark}\label{thm2} Let $\mathbb{I}$ be the unit tensor of order $m\ge3$ and dimension $n\ge2$.
Suppose that $P$ and $Q$ are two matrices of order $n$ with $P\mathbb{I}Q=\mathbb{I}$, then $P, Q\in \mathbb{P}_n$.
\end{remark}

\begin{remark} By  Definition \ref{defn11}, we know that Theorem \ref{thm21} gives a characterization for the similarities of tensors with order $m\ge3$ dimension $n$, i.e., as for the similarity of tensors, we need only consider the permutation similarity, the diagonal similarity and their compositions.
\end{remark}

\begin{theorem}\label{thm22}
 Let $\mathbb{A}$ and $\mathbb{B}$ be two order $m\geq 3$ dimension $n$ tensors. If  the tensors $\mathbb{A}$ and $\mathbb{B}$ are similar,
 then there exists a tensor $\mathbb{C}$ such that $\mathbb{A}$ and $\mathbb{C}$ are diagonal similar, and $\mathbb{B}$ and $\mathbb{C}$ are permutational similar.
\end{theorem}
\begin{proof}
By Definition \ref{defn11} and Theoren \ref{thm21}, there exist a permutation matrix $R$ and an invertible diagonal matrix $D$ of order $n$
such that $Q=DR$,  $P=R^TD^{1-m}$  and $\mathbb{B}=P\mathbb{A}Q$.

Take $\mathbb{C}=D^{1-m}\mathbb{A}D$, then $\mathbb{B}=R^T\mathbb{C}R=R^T\mathbb{C}(R^T)^T$.
By  Definitions \ref{defn12} and \ref{defn13}, the results hold.
 \end{proof}

\section{Some applications}
\hskip.6cm Let $Z(\mathbb{A})$ be the tensor obtained by replacing all the nonzero entries of $\mathbb{A}$ by one. Then $Z(\mathbb{A})$ is called the zero-nonzero pattern (or simply the zero pattern) of $\mathbb{A}$. Let $a$ be a complex number, we define $Z(a)=1$ if $a\not=0$ and $Z(a)=0$ if $a=0$.

\begin{lemma}\label{lem31}
Let $\mathbb{A}=(a_{i_1i_2\ldots i_m})_{1\leq i_j\leq n \hskip.12cm (j=1, \ldots, m)}$ and $\mathbb{B}=(b_{i_1i_2\ldots i_m})_{1\leq i_j\leq n \hskip.1cm (j=1, \ldots, m)}$ be two order $m\geq 3$ dimension $n$ tensors. If  the tensors $\mathbb{A}$ and $\mathbb{B}$ are diagonal similar,
then $Z(\mathbb{A})$=$Z(\mathbb{B})$.
\end{lemma}
\begin{proof}
By Definition \ref{defn13}, there exists  an invertible diagonal matrix $D=diag(d_{11}, \ldots, d_{nn})$ of order $n$ such that $\mathbb{B}=D^{1-m}\mathbb{A}D$.
Then

$b_{i_1i_2\ldots i_m}=(D^{1-m}\mathbb{A}D)_{i_1i_2\ldots i_m}$

\hskip1.4cm $=\sum\limits_{j_1,\ldots, j_m=1}^{n}(D^{1-m})_{i_1j_1}a_{j_1j_2\ldots j_m}d_{j_2i_2}\ldots d_{j_mi_m}$

\hskip1.4cm $=a_{i_1i_2\ldots i_m}d^{1-m}_{i_1i_1}d_{i_2i_2}\ldots d_{i_mi_m}.$

\noindent Therefore $b_{i_1i_2\ldots i_m}\not=0\Leftrightarrow a_{i_1i_2\ldots i_m}\not=0$, and thus $Z(\mathbb{A})$=$Z(\mathbb{B})$.
\end{proof}

\begin{theorem}\label{thm32}
Let $\mathbb{A}$ and $\mathbb{B}$ be two order $m\geq 3$ dimension $n$ tensors. If  the tensors $\mathbb{A}$ and $\mathbb{B}$ are  similar,
then $Z(\mathbb{A})$ and $Z(\mathbb{B})$ are permutational similar.
\end{theorem}
\begin{proof}
By Theorem \ref{thm22}, there exists a tensor $\mathbb{C}$ such that $\mathbb{A}$ and $\mathbb{C}$ are diagonal similar, and $\mathbb{B}$ and $\mathbb{C}$ are permutational similar. It is easy that $Z(\mathbb{A})=Z(\mathbb{C})$ by Lemma \ref{lem31}.
  Now we show that $Z(\mathbb{B})$ and $Z(\mathbb{C})$ are permutational similar.

Let $R$ be a permutation matrix of order $n$ such that  $\mathbb{B}=R\mathbb{C}R^T$. Then

  $(Z(\mathbb{B}))_{i_1i_2\ldots i_m}=(Z(R\mathbb{C}R^T))_{i_1i_2\ldots i_m}$

\hskip2.5cm $=Z(\sum\limits_{j_1,\ldots, j_m=1}^{n}r_{i_1j_1}c_{j_1j_2\ldots j_m}(R^T)_{j_2i_2}\ldots (R^T)_{j_mi_m})$

\hskip2.5cm $=Z(\sum\limits_{j_1,\ldots, j_m=1}^{n}r_{i_1j_1}c_{j_1j_2\ldots j_m}r_{i_2j_2}\ldots r_{i_mj_m})$

\hskip2.5cm $=Z(c_{\sigma(i_1)\sigma(i_2)\ldots \sigma(i_m)})$

\hskip2.5cm $=\sum\limits_{j_1,\ldots, j_m=1}^{n}r_{i_1j_1}Z(c_{j_1j_2\ldots j_m})r_{i_2j_2}\ldots r_{i_mj_m}.$

\noindent So $Z(\mathbb{B})=RZ(\mathbb{C})R^T$,  and thus $Z(\mathbb{A})$ and $Z(\mathbb{B})$ are permutational similar.
\end{proof}

\begin{corollary}\label{cor33}
Let $\mathbb{A}$ and $\mathbb{B}$ be two order $m\geq 3$ dimension $n$ tensors, $N(\mathbb{A})$  the number of the nonzero entries of tensor $\mathbb{A}$.
If  the tensors $\mathbb{A}$ and $\mathbb{B}$ are  similar, then $N(\mathbb{A})=N(\mathbb{B})$.
\end{corollary}
\begin{proof}
By Theorem \ref{thm32}, there exists a   permutation matrix $R=R_{\sigma}=(r_{ij})$ of order $n$  such  that  $Z(\mathbb{B})=RZ(\mathbb{A})R^T$,
where $\sigma\in S_n$ is a  permutation on the set $[n]$ and $r_{ij}=1\Leftrightarrow j=\sigma(i)$. Then similar to the proof of Theorem \ref{thm32}, we have

$(Z(\mathbb{B}))_{i_1i_2\ldots i_m}=(RZ(\mathbb{A})R^T)_{i_1i_2\ldots i_m}$

\hskip2.5cm $=\sum\limits_{j_1,\ldots, j_m=1}^{n}r_{i_1j_1}(Z(\mathbb{A}))_{j_1j_2\ldots j_m}(R^T)_{j_2i_2}\ldots (R^T)_{j_mi_m}$

\hskip2.5cm $=\sum\limits_{j_1,\ldots, j_m=1}^{n}r_{i_1j_1}(Z(\mathbb{A}))_{j_1j_2\ldots j_m}r_{i_2j_2}\ldots r_{i_mj_m}$

\hskip2.5cm $=(Z(\mathbb{A}))_{\sigma(i_1)\sigma(i_2)\ldots \sigma(i_m)}.$

\noindent Therefore $(Z(\mathbb{B}))_{i_1i_2\ldots i_m}\not=0\Leftrightarrow (Z(\mathbb{A}))_{\sigma(i_1)\sigma(i_2)\ldots \sigma(i_m)}\not=0$, and thus $N(\mathbb{A})$=$N(\mathbb{B})$.
\end{proof}

\begin{remark}\label{rem34}
Note that the result of Corollary \ref{cor33} does not hold when $m=2$. For example,
let $ P=\left(
  \begin{array}{cc}
    1 & 1 \\
    1 & 0\\
  \end{array}
\right), $
$Q=\left(
  \begin{array}{cc}
    1 & -1 \\
    1 & 0\\
  \end{array}
\right),$   $A=\left(
  \begin{array}{cc}
    0 & 0 \\
    1 & 0\\
  \end{array}
\right),$  $B=PAQ=\left(
  \begin{array}{cc}
    1 & -1 \\
    1 & -1\\
  \end{array}
\right).$ It is clear $PIQ=QIP=I$, then the matrix $A$ and $B$ are similar, but $N(A)=1\not=N(B)=4.$
\end{remark}

\begin{definition}{\rm (\cite{HHLQ13})}\label{defn35}
Let  $\mathbb{A}=(a_{i_1i_2\ldots i_m})_{1\leq i_j\leq n \hskip.12cm (j=1, \ldots, m)}$ be an order $m\geq 3$ dimension $n$ tensor.
If $ a_{i_1i_2\ldots i_m} \equiv 0$ whenever $\min\{i_2, \ldots, i_m\}<i_1$, then $\mathbb{A}$ is called an upper triangular tensor.
If $ a_{i_1i_2\ldots i_m}\equiv 0$ whenever $\max\{i_2, \ldots, i_m\}>i_1$, then $\mathbb{A}$ is called a low triangular tensor.
If $\mathbb{A}$ is either upper or low triangular tensor, then $\mathbb{A}$ is called a  triangular tensor.
If $a_{i_1i_2\ldots i_m}\equiv 0$ whenever $i_1i_2\ldots i_m\not=i_1i_1\ldots i_1$, then $\mathbb{A}$ is called a   diagonal tensor.
%If for  any $i_j\not=i_k$ where $j,k\in [n]$, $a_{i_1i_2\ldots i_m}\equiv 0$, then $\mathbb{A}$ is called a   diagonal tensor.
\end{definition}

Clearly,  a  triangular tensor is  both an upper and a low triangular tensor.

\begin{corollary}\label{cor36}
Let $\mathbb{A}$ and $\mathbb{B}$ be two order $m\geq 3$ dimension $n$ tensors.
Suppose that $\mathbb{A}$ and $\mathbb{B}$ are  similar,  and $\mathbb{A}$  is a  diagonal tensor,
then $\mathbb{B}$ is a diagonal tensor.
\end{corollary}
\begin{proof}
We only need to show $(Z(\mathbb{B}))_{i_1i_2\ldots i_m}=0$ (and thus $b_{i_1i_2\ldots i_m}=0)$
for any $i_1i_2\ldots i_m\not=i_1i_1\ldots i_1$ where $i_1\in[n]$. %$i_j\not=i_k$ where $j,k\in [n]$.

By Theorem \ref{thm32}, there exists a   permutation matrix $R=R_{\sigma}=(r_{ij})$ of order $n$  such  that  $Z(\mathbb{B})=RZ(\mathbb{A})R^T$,
where $\sigma\in S_n$ is a  permutation on the set $[n]$ and $r_{ij}=1\Leftrightarrow j=\sigma(i)$. Then
by the proof of Corollary \ref{cor33},
 we have $(Z(\mathbb{B}))_{i_1i_2\ldots i_m}\not=0\Leftrightarrow (Z(\mathbb{A}))_{\sigma(i_1)\sigma(i_2)\ldots \sigma(i_m)}\not=0$.
 Note that  $i_1i_2\ldots i_m\not=i_1i_1\ldots i_1$ if and only if $\sigma(i_1)\sigma(i_2)\ldots\sigma(i_m)\not=\sigma(i_1)\sigma(i_1)\ldots\sigma(i_1)$,
 therefore $(Z(\mathbb{B}))_{i_1i_2\ldots i_m}=0$ when $i_1i_2\ldots i_m\not=i_1i_1\ldots i_1$  since $\mathbb{A}$ is a  diagonal tensor.
 % Therefore, for any $i_j\not=i_k$ where $j,k\in [n]$, $\sigma(i_j)\not=\sigma(i_k)$,
%and thus $(Z(\mathbb{B}))_{i_1i_2\ldots i_m}=0$ by $(Z(\mathbb{A}))_{\sigma(i_1)\sigma(i_2)\ldots \sigma(i_m)}=0$ since $\mathbb{A}$ is a  diagonal tensor.
\end{proof}

\begin{remark}\label{rem37}
Note that the result of Corollary \ref{cor36} does not hold when $m=2$.
For example,  all real symmetric matrices  are similar to   diagonal matrices, but symmetric matrices are not diagonal matrices in general.
Another example,  if a real matrix $A$ of order $n$ has $n$ distinct eigenvalues, then $A$ is similar to  diagonal matrices,
but $A$ is  not necessarily a diagonal matrix.
\end{remark}

\begin{definition}
A Jordan block $J_k(\lambda)$ is an upper triangular matrix of order $k$ of the form
$$ P=\left(
  \begin{array}{ccccc}
    \lambda & 1 &    &    &   \\
            &  \lambda &  1 &   & \\
              &   & \ddots & \ddots &\\
              &    &  &\ddots & 1\\\
              &  &   &  & \lambda
  \end{array}
\right). $$
There are $k-1$ terms 1 in the superdiagonal; the scalar $\lambda$ appears $k$ times on the main diagonal. All other entries are zero, and $J_1(\lambda)=[\lambda]$.
\end{definition}

\begin{theorem}{\rm(\cite{HJ2005}, Jordan canonical form theorem)}
Let $A$ be a given complex matrix of order $n$. There is a nonsingular matrix $S$ of order $n$ such that
$A=SJS^{-1}$, where
$$J=\left(
  \begin{array}{ccccc}
    J_{n_1}(\lambda_1) &  &    &  0  \\
            &  J_{n_2}(\lambda_2) &  &    \\
              &   & \ddots & \\
             0 &    &  &J_{n_k}(\lambda_k)
              \end{array}
\right) $$
 \noindent is a Jordon matrix with $n_1+n_2+\ldots+n_k=n$, and $J$ is unique up to permutations of the diagonal Jordan blocks, called the Jordan canonical form of $A$. The eigenvalues $\lambda_i, i=1,\ldots, k$ are not necessarily distinct. If $A$ is a real matrix with only real eigenvalues, then the similarity matrix $S$ can be taken to be real.
\end{theorem}

Clearly, the Jordon  matrix $J$ is an upper triangular matrix.

\begin{remark}\label{rem310}
The results of matrices about the Jordan canonical form  cannot be extended to tensors since not all tensors are similar to some upper triangular tensor by Corollary \ref{cor33}.
\end{remark}

\end{document}